\documentclass[12pt]{amsart} 
\usepackage{amssymb} 
\usepackage{amsfonts} 
\usepackage{amsmath} 
\usepackage{amsthm} 
\usepackage{array} 
\usepackage{geometry} 
\usepackage{graphicx} 
\usepackage{mathrsfs} 
\usepackage{pifont} 
\usepackage[all]{xypic}
\usepackage{verbatim} 
\usepackage{stmaryrd}
\usepackage{color}

\numberwithin{equation}{section}

\theoremstyle{plain} 
\newtheorem{thm}{Theorem}[subsection] 
\newtheorem{cor}[thm]{Corollary} 
\newtheorem{lem}[thm]{Lemma} 
\newtheorem{prop}[thm]{Proposition} 
\theoremstyle{definition} 
\newtheorem{defn}[thm]{Definition} 
\newtheorem{defn*}{Definition}[section] 
\newtheorem{prop*}{Proposition}[section] 
\theoremstyle{remark} 
\newtheorem*{rem}{Remark} 
\newcommand{\N}{\mathbb{N}} 
\newcommand{\Q}{\mathbb{Q}} 
\newcommand{\C}{\mathbb{C}} 
\newcommand{\A}{\mathbb{A}}     
\newcommand{\G}{\mathbb{G}} 
\newcommand{\sA}{\mathcal{A}}
\newcommand{\sB}{\mathcal{B}}

\newcommand{\sF}{\mathcal{F}}

\newcommand{\sI}{\mathcal{I}}
\newcommand{\sK}{\mathcal{K}}
\newcommand{\sM}{\mathcal{M}}
\newcommand{\sN}{\mathcal{N}}
\newcommand{\sO}{\mathcal{O}}

\newcommand{\sT}{\mathcal{T}}
\newcommand{\sX}{\mathcal{X}}

\newcommand{\vir}{\mathrm{vir}}
\newcommand{\virclass}[1]{{[{#1}]^\vir}}            

\begin{document}
\title{The Evaluation Space of Logarithmic Stable Maps}


\author{Dan Abramovich}

\author{Qile Chen}

\author{William Gillam}

\author{Steffen Marcus}

\thanks{Research of Abramovich, Chen supported in part by funds from NSF award 0603284. Gillam supported by an NSF postdoctoral fellowship.  Marcus supported by an NSERC PGS-D doctoral fellowship.}

\address{Department of Mathematics\\
Brown University\\
Box 1917\\
Providence, RI 02912\\
U.S.A.}
\email{abrmovic, q.chen, wgillam, ssmarcus@math.brown.edu}

\date{\today} 
\begin{abstract} The evaluation stack $\wedge X$ for minimal logarithmic stable maps is constructed, parameterizing families of standard log points in the target log scheme.  This construction provides the ingredients necessary to define appropriate evaluation maps for minimal log stable maps and establish the logarithmic Gromov-Witten theory of a log-smooth Deligne-Faltings log scheme.
\end{abstract}
\maketitle
\setcounter{tocdepth}{1}
%
%
\section{Introduction}

\subsection{Main results}
Logarithmic Gromov-Witten (GW) theory was first proposed during a 2001 workshop lecture by Bernd Siebert  \cite{S01}. It provides an approach to vastly generalize {\em relative} GW theory, in which enumerative invariants of curves on varieties satisfying certain contact conditions are defined and used for computing usual GW invariants through degenerations. Current relative GW theory was introduced in the symplectic setting by An-Min Li and Yongbin Ruan \cite{LR01}, as well as in parallel work of Eleny Ionel and Thomas Parker \cite{IP03,IP04}. It was recast in algebraic geometry by Jun Li \cite{Li01,Li02}. 

 The aim of this paper is to continue the development of {\em logarithmic} GW theory along the lines of \cite{K09}, \cite{C10}, and \cite{AC10}, where Kontsevich stable maps to log-schemes are introduced and studied.  The next steps we take here are the following:

In Definition \ref{def:wedgeprime} we define {\em families of  standard log points} in a fine saturated log scheme $(X,\sM_X)$. This is a category we denote $\wedge'X$ fibered over the category $\mathfrak{LogSch}^{\text{fs}}$  of fine and saturated log schemes. The first main result is the following: 

\begin{thm}\label{thm:main} 
There is a logarithmic algebraic stack $(\wedge X, \sM_{\wedge X})$ representing $\wedge'X$.
\end{thm}

In case $X$ is projective and $\sM_X$ is a  Deligne--Faltings log structure (Definition \ref{defn:DF}), it is shown in \cite{C10} and \cite{AC10} that there is a proper logarithmic Deligne--Mumford (DM) stack $(\sK_\Gamma(X,\sM_X),\sM_{\sK_\Gamma(X,\sM_X)})$ of logarithmic stable maps into $X$ with numerical characteristics $\Gamma=(g,n,c_i,\beta)$. Mark Gross and Bernd Siebert have recently announced a parallel theory of basic log maps that is expected to work in even more general situations. Here $g,n,\beta$ are as usual the genus, the number of marked points, and curve class, and $c_i$ are prescribed contact orders.

In Definition \ref{def:eval} we construct natural morphisms 
$$\text{ev}_i:\sK_\Gamma(X,\sM_X) \to \wedge X,\quad i=1,\ldots,n,$$ 
called {\em logarithmic evaluation maps}. In case $(X,\sM_X)$ is log smooth, we construct in Proposition \ref{prop:vfc} a natural virtual fundamental class
$$\left[\sK_\gamma(X,\sM_X)\right]^{\text{vir}}$$ 
on $\sK_\Gamma(X,\sM_X)$, which finally enables us to define {\em  logarithmic Gromov--Witten invariants} (Definition \ref{def:gw-inv}):
    \[
    \big<\gamma_1,\ldots,\gamma_n\big>^{(X,\sM_X)}_{\Gamma} :=
    \left(\prod_{i=1}^n \text{ev}_i^\ast\gamma_i \right)\cap
    \left[\sK_\gamma(X,\sM_X)\right]^{\text{vir}}. 
    \]

\subsection{Conventions and notation} We will always be working over the base field $\C$.  Unless otherwise specified, our neighborhoods will always be taken in the \'etale topology, so that by locally we will mean \'etale locally.  The algebraic closure of a point $p$ is denoted by $\bar p$.

We will be using the conventions of logarithmic geometry in the sense of Kato-Fontaine-Illusie, and will assume familiarity with the theory up to at least \cite{K89}. 

 Denote a log scheme or stack by the pair $(X,\exp_X:\sM_X\to\sO_X)$ where $X$ is the underlying space and $\exp_X:\sM_X\to \sO_X$ is the structure map of the associated log structure on $X$.  The notation $(X,\sM_X)$ will be used when no confusion arises.  Denote by $\overline{\sM_X}=\sM_X/\sO_X^\ast$ the characteristic of $(X,\sM_X)$.  A log scheme or log stack $X$ appearing without a log structure is assumed to be taken with the trivial log structure $\sO_X^\ast$.
 
 Denote by $\mathfrak{Sch}$, $\mathfrak{LogSch}$, and  $\mathfrak{LogSch}^{\text{fs}}$ the categories of schemes, log schemes, and fine saturated log schemes respectively.  All log schemes will be assumed to be objects in $\mathfrak{LogSch}^{\text{fs}}$ unless otherwise noted.

\subsection{Logarithmic stable maps}
	Inspired by Siebert's original lecture and recent successes in the use of log geometry to compactify moduli spaces, such as \cite{O08}, the moduli of Kontsevich log-stable maps was taken up in \cite{K09} and again in \cite{C10,AC10} with a view towards developing logarithmic GW theory in the setting where the relative divisors $D_i$ are simple normal crossings.	

Let $(X,\sM_X)$ be a log-smooth, fine, saturated log scheme where $X$ is a projective variety and its log structure $\sM_X$ is the divisorial log structure corresponding to a simple normal crossing divisor $D\subset X$.  The moduli stack of \emph{minimal} log stable maps $\sK_\Gamma(X,\sM_X)$ parameterizes families of maps $f: (C,\sM_C)\to (X,\sM_X)$ from log-smooth curves where the underlying map is stable in the usual sense.  Furthermore, the maps are required to satisfy an additional minimality condition necessary to control the log structures associated to the maps and ensure properness of the stack.  The notation $\Gamma$ collects the discrete data of the map such as genus, number of marked points, curve class and contact orders. 

\begin{rem}
The term `minimality' in this context originates from \cite{K09}. It is a phenomeon derived from the following general question: given a category $\sF$ fibered in groupoids over $\mathfrak{LogSch}$, when does there exist an algebraic stack $F$ fibered in groupoids over $\mathfrak{Sch}$ such that, when endowed with a log structure $\sM_F$, there is an equivalence $(F,\sM_F)\cong\sF$ of groupoid fibrations over $\mathfrak{LogSch}$?  
This categorical question produces a nice categorical framework for minimality which can be completely described and will be made available in a subsequent paper.  
\end{rem}

The stack $\sK_\Gamma(X,\sM_X)$ is a proper Deligne--Mumford stack and
comes equipped with a natural log structure
$\sM_{\sK_\Gamma(X,\sM_X)}$, dictated by the minimality condition,
making the pair a log algebraic stack.  There is a universal
log-smooth curve $(\mathfrak{C},\sM_{\mathfrak{C}})$ fitting into the
following universal diagram: 
\[
\xymatrix{
(\mathfrak{C},\sM_{\mathfrak{C}}) \ar[r]^f \ar[d] & (X,\sM_X)\\
(\sK_\Gamma(X,\sM_X),\sM_{\sK_\Gamma(X,\sM_X)}) \ar[d]& \\
\sK_\Gamma(X,\sM_X).&\\}
\]
As a fibered category over $\mathfrak{LogSch}^\text{fs}$, the log stack $(\sK_\Gamma(X,\sM_X),\sM_{\sK_\Gamma(X,\sM_X)})$ parameterizes logarithmic stable maps to $(X,\sM_X)$ over a base log scheme with an arbitrary log structure.  A morphism $S\to\sK_\Gamma(X,\sM_X)$ from a scheme $S$ corresponds precisely to a strict morphism $(S,\sM_S)\to (\sK_\Gamma(X,\sM_X),\sM_{\sK_\Gamma(X,\sM_X)})$, giving an extended diagram
 \[
\xymatrix{
(C^{min},\sM_C^{min}) \ar[r] \ar@/^1pc/[rr]^{f'} \ar[d] & (\mathfrak{C},\sM_{\mathfrak{C}}) \ar[r]\ar[d] & (X,\sM_X)\\
(S^{min},\sM_S^{min}) \ar[d] \ar[r]& (\sK_\Gamma(X,\sM_X),\sM_{\sK_\Gamma(X,\sM_X)}) \ar[d]& \\
S\ar[r] & \sK_\Gamma(X,\sM_X). &\\}
\]
In this way $\sK_\Gamma(X,\sM_X)$ is a stack over $\mathfrak{Sch}$ parameterizing minimal log stable maps.

\subsection{Evaluation spaces} 
The log structure $\sM_C$ on the source curve of a minimal log stable map prescribes not only the standard log points on $C$ giving the marked points of the map, but also contact orders for each of these points.  This extra data is pulled back from the target log-structure $\sM_X$ through the log-map $f:(C,\sM_C)\to (X,\sM_X)$.  It is a situation quite similar to that of twisted stable maps \cite{AV02,AGV08}, where marked points are endowed with a $B\mu_m$ stack structure and their evaluations in a target stack $\sX$ are studied using the rigified cyclotomic inertia stack $\overline{\sI}_\mu(\sX)$.  The inertia stack is constructed by parameterizing families of maps from $B\mu_m$ into the target.  To construct the logarithmic evaluation space, we parameterize families of log-maps from the standard log point into $(X,\sM_X)$.  

The theory of minimal log stable maps in \cite{C10,AC10} is outlined for target schemes with a Deligne-Faltings log structure.  Much use is made of the moduli of such log structures and section \ref{s:modlog} is devoted to this topic.  Our main construction begins in section \ref{sec:logpts}, covering the case where $\sM_X$ is a Deligne-Faltings log structure corresponding to a single line bundle with a section.  
A notion of minimality is introduced to identifiy the necessary families to include in order to build a stack $\wedge X$ over the $\mathfrak{Sch}$.  The bulk of this section is devoted to the technicalities of the construction of $\wedge X$, showing it is algebraic, and identifying properties which make it useful for log GW theory.  Section \ref{ss:dfncase} generalizes the construction to any fine, saturated log structure. Evaluation maps $\text{ev}_i:(\sK_\Gamma(X,\sM_X),\sM_{\sK_\Gamma(X,\sM_X)})\to(\wedge X,\sM_{\wedge X})$
are defined in section \ref{s:loggw}, as are the GW invariants produced in this setting.  


\section{Moduli of log structures}\label{s:modlog}

The construction of the evaluation stack $\wedge X$ presented below goes by way of the moduli of Deligne-Faltings log structures, that is, log structures corresponding to line bundles with sections.  To begin, we remind the reader of Olsson's stacks of log structures, as these also play an important role in our constructions.

\subsection{Olsson's stacks of log structures} \label{ss:logstack}
Let $(S,\sM_S)$ be a fine log scheme (not necessarily saturated).  In \cite{O03}, Martin Olsson constructs the algebraic stack $\mathcal{L} og_{(S,\sM_S)}$ parameterizing fine log structures.  As a category fibered over the category of $S$-schemes $\mathfrak{Sch}_S $, it is defined as follows:
\begin{enumerate}
\item[(a)] an object of $\mathcal{L} og_{(S,\sM_S)}$ over the $S$-scheme $f:Y\to S$ consists of a fine log structure $\sM_Y$ on $Y$ and a morphism $f^\flat: f^\ast\sM_S\to\sM_Y$;  
\item[(b)] a morphism over $X\to Y$ consists of a \emph{strict} morphism $(X,\sM_X)\to (Y,\sM_Y)$ over $(S,\sM_S)$. 
\end{enumerate}
The data of an object is equivalent to a morphism of fine log schemes $f:(Y,\sM_Y)\to(S,\sM_S)$ extending $f$.  This stack is algebraic (though not quasi-seperated), and locally of finite presentation over $S$.

When $\sM_S$ is also saturated, the open substack $\mathcal{T}or_{(S,\sM_S)}$ in $\mathcal{L} og_{(S,\sM_S)}$ parameterizes fine, saturated log structures.  Given an $S$-scheme $f:Y\to S$, a morphism $Y\to  \mathcal{T}or_{(S,\sM_S)}$ corresponds to a fine, saturated log structure $\sM_Y$ and a log morphism $f:(Y,\sM_Y)\to(S,\sM_S)$.  There is a forgetful map $\text{Log}:\mathcal{T}or_{(S,\sM_S)}\to\mathcal{T}or_{\C}$ forgetting all the data but the log structure $\sM_Y$.  If $(S,\sM_S)$ is log smooth, then the map Log is a smooth representable morphism of algebraic stacks.  For our construction of $\wedge X$ in section \ref{ss:wedgex}, we will need to consider $\mathcal{T}or_{(X,\sM_X)}$ for our target log scheme $(X,\sM_X)$.

\subsection{$DF(n)$ log structures} \label{ss:dfn-log-str} The divisorial log structure associated to a divisor $D\subset X$ is defined on a neighbourhood $U$ by
\[ \
\sM_{D\subset X}(U):= \left\{ f\in\sO_X(U): f \text{ is invertible on } (U\setminus D)\right\}.
\]
This sheaf of monoids motivates the relationship between relative and logarithmic Gromov-Witten theory.  The case when $D$ is a simple normal crossings divisor corresponds to the divisorial log structure $\sM_X$ on $X$ admitting a morphism of sheaves $\N^r\to \overline{\sM}
_X$ that locally lifts to a chart.  Following Kato, we call locally free log structures with this property Deligne-Faltings log structures.

\begin{defn}\label{defn:DF}
Let $S$ be a scheme.  A log structure $\sM_S$ on $S$ is called a \emph{Deligne-Faltings} (DF) \emph{Log Structure} if there exists a morphism of sheaves of monoids $\N^r\to\overline{\sM_S}$ lifting locally to a chart $\N^r\to\sM_S$. A DF log structure $\sM_{S}$ is \emph{generic} if it is nontrivial and the map $\N^{r}\to \overline{\sM}_{S}$ is an isomorphism on every geometric point.  We define the \emph{rank} of a DF log structure $\sM_S$ to be the integer $r:=\text{max}\{k:{\overline{\sM}_S}_{\bar p}\cong \N^k\}$ where the maximum is taken over all geometric points $\bar p\in S$.  For notational convenience, we call a DF log structure of rank at most $n$ a DF($n$) \emph{log structure}.  
\end{defn}

The data of a DF($n$) log structure is equivalent to a collection of $n$ line bundles with sections $\{(L_i, s_i)\}_{i=1}^n$ (see \cite{K89}, Complement 1).  One direction is fairly obvious.  Let $e_i$ be the $i$'th standard generator of $\N^n$, and let  $\beta:\N^n\to \overline{\sM_S}$ be the global map with local liftings $\tilde \beta:\N^n\to \sM_S$.  Let $\pi:\sM_S\to\overline{\sM_S}$ be the quotient map.  The pre-image $\pi^{-1}(\beta (e_i))$ is an $\sO_S^\ast$-torsor, which corresponds to a line bundle on $S$, say $L_i$.  The map $\pi^{-1}(\beta (e_i))\subset\sM_S\overset{\alpha}{\to}\sO_S$ determines a section $s_i:L_i\to\sO_S$ of $\Gamma(X,L^\vee)$.  Thus the $DF(n)$ log structure $\sM_S$ gives the data of the pairs $\{L_i,s_i\}_{i=1}^{n}$.  In fact, such a collection is sometimes taken to be the definition of a DF log structure elsewhere in the literature (see \cite{K89}), with rank then defined to be the number of line bundles $n$.  Our notion of rank in Definition \ref{defn:DF} is a bit more subtle, and depends on the zero loci of the sections $s_i$.   

We will in particular be interested in parameterizing DF(1) log
structures.  Recall that a log point is a log scheme $(\text{Spec
}\C,\C^\ast\oplus P)$, where $P$ is a monoid and the structure map
$\alpha:\C^\ast\oplus P\to\sO_{\text{Spec } \C}$ is given by sending
$a\oplus \overline{0}$ to $a$ and everything else to $0$.  When $P=\N$
we call this the \emph{standard log point}.

\begin{defn}\label{Def:standard-log-point}
A \emph{family of DF(1) log structures} over a log scheme $(S,\sM_S)$ is a morphism of log schemes $(f,f^\flat):(S,\sM_{S}')\to(S,\sM_S)$ such that the morphism $f:S\to S$ on the underlying schemes is the identity and $\sM_{S'}=\sM_S\oplus_{\sO_{S}^{*}}\sN$ where $\sN$ is a DF log structure given by a map $\N\to\overline\sN$.  A family of DF(1) log structures $(f,f^\flat):(S,\sM_{S'})\to(S,\sM_S)$ is a \emph{family of standard log points} if in particular $\sN$ is a generic rank 1 DF log structure.
\end{defn}

\subsection{$[\A^1/\G_m]$ and $B\G_m$ as classifying stacks of DF($1$) log structures} \label{ss:families} The stack $[\A^1/\G_m]$ with the trivial log structure can be interpreted as parameterizing families of DF(1) log structures.  Here the quotient is taken with respect to the multiplication action.  We quickly outline this example, as this interpretation plays an important role in the constructions to follow.

The stack $B\G_m$ sits inside the quotient $\sA:=[\A^1/\G_m]$ as the origin.  Let $i:B\G_m\to\sA$ be the inclusion map.  A morphism $X\to\sA$ from a scheme $X$ corresponds to a principal $\G_m$-bundle $P\to X$ with a $\G_m$-equivariant map $P\to \A^1$:
\[
\xymatrix{
P \ar[r] \ar[d] & \A^1 \ar[d]^\pi\\
X \ar[r] & \sA.\\}
\]
The $\G_m$-bundle $P$ extends uniquely to a line bundle $L\to X$, and the equivariant map $P\to\A^1$ determines a morphism of line bundles $L\to\sO_X$, that is, a section $s\in\Gamma(X,L^\vee)$.  This data is equivalent to a DF(1) log structure on $X$.  This process is certainly reversible, and in this way $\sA$ classifies DF(1) log structures (see \cite{O03} Example 5.13).

The map $\pi:\A^1\to\sA$ gives a universal line bundle with universal section $s$ determined by the multiplication action of $\G_m$.  Thus $\sA$ is naturally equipped with the DF log structure $\sM_\sA$ given by $(\A^1,s)$.  Consider the family of DF(1) log structures given by the morphism $g:(\sA,\sM_{\sA})\to\sA$ which is the identity on $\sA$ and the inclusion $g^\flat:g^\ast\sO_\sA^\ast\to \sM_\sA$ on log structures.  We exhibit this as the universal family of DF(1) log structures over $\sA$.

The data of a morphism $(S,\sM_S)\to \sA$ where $(S,\sM_S)$ is a log scheme is equivalent to giving a family of DF(1) log sctructures over $(S,\sM_S)$.  A morphism of log schemes $h:(S,\sM_S)\to \sA$ corresponds to a cartesian diagram, 
\[
\xymatrix{
(S,\sM_S') \ar[r]^{h'} \ar[d] & (\sA,\sM_A) \ar[d]^g\\
(S,\sM_S) \ar[r]_h & \sA}
\]
in the category $\mathfrak{LogSch}^\text{fs}$.  The sheaf of monoids $\sM_S'$ is the log structure associated to the pushout of $\sM_S$ with $h^{-1}\sM_\sA$.  Thus $\sM_S'=\sM_S\oplus\sN$ where, by our discussion above and the definition of $\sM_\sA$, $\sN$ is the DF(1) log structure given by the morphism $S\to\sA$.  Inverting this, let $(S,\sM_S\oplus \sN)\to(S,\sM_S)$ be a given family of DF(1) log structures.  The DF(1) log structure $\sN$ corresponds to a morphism $S\to\sA$, and this determines the morphism $(S,\sM_S)\to\sA$.

Thus, the log stack $\sA$ represents the fibered category over $\mathfrak{LogSch}$ parametrizing families of rank 1 DF log structures.  As a category fibered over $\mathfrak{Sch}$, the stack $B\G_m$ parameterizes line bundles with the zero section.  Through its inclusion, $B\G_m$ has the log structure $\sM_{B\G_m}:=i^\ast\sM_{\sA}$ induced by restriction. When viewed as a universal family $(B\G_m,\sM_{B\G_m})$ over $B\G_m$, a morphism of log schemes $(S,\sM_S)\to B\G_m$ is equivalent to a family of standard log points over $(S,\sM_S)$.  Thus, the log stack $B\G_m$ is similarly the fibered category over $\mathfrak{LogSch}$ parametrizing families of standard log points.

\subsection{Families of Standard Log Points in a Log Scheme $(X,\sM_X)$}\label{ss:wedgeprime}

Fix a fine, saturated log scheme $(X,\sM_X)$.  We now describe the category of standard log points in $(X,\sM_X)$ fibered over the category $\mathfrak{LogSch}^{\text{fs}}$. 

\begin{defn}\label{def:wedgeprime}
Define a category $\wedge' X$ fibered over the category $\mathfrak{LogSch}^{\text{fs}}$ as follows:
\begin{enumerate}
\item[(a)] An object of $\wedge' X(S,\sM_S)$ over a log scheme $(S,\sM_S)$ consists of a family of standard log points $(S,\sM_{S}')\to (S,\sM_{S})$ with a morphism $\phi:(S,\sM_{S}')\to (X,\sM_X)$,
\begin{equation*}\label{logpt:log-pt-defn}
\xymatrix{
(S,\sM_{S}') \ar[r]^\phi \ar[d] & (X,\sM_X)\\
(S,\sM_S). & \\}
\end{equation*}
\item[(b)] An arrow consists of a morphism of families of standard log points $F:(S,\sM_{S}') \to (T,\sM_{T}')$ over some $f:(S,\sM_S)\to (T,\sM_T)$ forming a cartesian square and making the following diagram commutative:
\begin{equation*}\label{equ:arrow-log-pt}
\xymatrix{
&&(X,\sM_X)\\
(S,\sM_{S}') \ar[r]^F \ar[d] \ar@/^/[urr]^{\phi_{S}}& (T,\sM_{T}') \ar@/_/[ur]_{\phi_{T}}\ar[d]\\
(S,\sM_S) \ar[r]_f & (T,\sM_T).\\}
\end{equation*}
\end{enumerate}
\end{defn}

Although it is certainly useful to parameterize such families over log schemes, we are especially interested in building the log algebraic stack $\wedge X$ over the category $\mathfrak{Sch}$ which, when equipped with a natural log structure and viewed as a category fibered over $\mathfrak{LogSch}^\text{fs}$, gives exactly $\wedge'X$.  We are led to a notion of minimality for families of standard log points in $(X,\sM_X)$.

\section{The stack of standard log points in $(X,\sM_X)$: DF(1) case} \label{sec:logpts}

In this section and the next we build the evaluation stack $\wedge X$.  Throughout, fix a target log scheme $(X,\sM_X)$ where $\sM_X$ is a DF($1$) log structure.  In section \ref{ss:dfncase} we will use a limit argument to generalize to the case of an arbitrary fine, saturated log scheme.

\subsection{Minimal Families}\label{ss:minimality}  We now introduce a notion of minimality for families of standard log points in $(X,\sM_X)$.  Minimality in this case can be described completely geometrically as a condition on a map of characteristic monoids.  This allows us to identify the appropriate families to parameterize when constructing $\wedge X$ in section \ref{ss:wedgex}.

\begin{defn}\label{defn:min}
A family of standard log points in $(X,\sM_X)$
\[ 
\xymatrix{
(S,\sM_{S}') \ar[d]\ar[r]^{\phi_S} & (X,\sM_X)\\
(S,\sM_S). &
}
\]
is called {\em minimal} if over each geometric point $s\in S$ the composition 
\[(\phi_S^\ast\overline{\sM}_{X})_s\to \overline{\sM}'_{S,s}\cong\overline{\sM}_{S,s}\oplus\overline{\sN}_s \to \overline{\sM}_{S,s}\] 
gives a surjection $(\phi_S^\ast\overline{\sM}_X)_s\to\overline{\sM}_{S,s}$. 
\end{defn}

\begin{prop}\label{prop:min-open}
Minimality is an open condition on $\wedge'X$.
\end{prop}
\begin{proof}
Consider a family of standard log points in $(X,\sM_X)$ as given in the above definition.  Assume that it is minimal at a point $s$ along the base $S$. Let $\delta$ generate the the stalk $(\phi_S^\ast\overline{\sM}_X)_s$ and let $e + \sigma$ be the image of $\delta$ through ${\phi_{S,s}^\flat}$, where $e$ and $\sigma$ are elements of $\overline{\sM}_{S,s}$, and $\sN_{s}$ respectively.  If we generalize to a nearby point of $s$, then since the specialization map is surjective on characteristics (see \cite{O03}, lemma 3.5 part iii), we have that $e+\sigma$ is trivial, making $e$ is trivial. This proves the statement.
\end{proof}

Denote by $D\subset X$ the locus along which the log structure $\sM_{X}$ is non-trivial. The following result provides for the existence of `enough' minimal families.

\begin{prop}\label{prop:min-cri}
For any family $\phi_S:(S,\sM_{S}')\to(X,\sM_X)$ of standard log points in $(X,\sM_X)$ over $(S,\sM_S)$, there exists a minimal family $\phi^{min}:(S,{\sM_S'}^{min})\to(X,\sM_X)$ over $(S,\sM_S^{min})$ and a map $f^{min}:(S,\sM_{S})\to (S,\sM_S^{min})$ fitting into a commutative diagram
 \[
\xymatrix{
&&(X,\sM_X)\\
(S,\sM_S')\ar@/^/[urr]^{\phi_S}\ar[r]\ar[d]&(S,{\sM_{S}'}^{min}) \ar@/_/[ur]_{\phi^{min}}\ar[d] & \\
(S,\sM_S)\ar[r]_{f^{min}}&(S,\sM_S^{min}) &}
\]  
where the square is cartesian in $\mathfrak{LogSch}^\text{fs}$.  Furthermore, the pair $(f^{min},\phi^{min})$ is unique up to a unique isomorphism.
\end{prop}
\begin{proof}
This is a local statement.  Shrinking $S$, we can assume there are global charts $\beta_{1}:\overline{\sM}_{S,s}\to \sM_{S}$ and $\beta_{2}:\overline{\sN}_{s}\cong \N\to \sN$ for some point $s\in S$. We assume also that $\phi_{S}(s)\in D$; the other case is similar and straightforward. Since $\sM_{X}$ is a DF(1) log structure, denote by $\delta$ the generator of $(\phi_S^\ast\overline{\sM}_{X})_s\cong\N$. As in the proof of the above proposition, let $e+\sigma$ denote the image of $\delta$ through the map $\phi_{S,s}^\flat:(\phi_S^\ast\overline{\sM}_{X})_s\to \overline{\sM}'_{S,s}\cong\overline{\sM}_{S,s}\oplus\overline{\sN}_s$, where $e\in \overline{\sM}_{S,s}$ and $\sigma\in \overline{\sN}_s$. Choose as our minimal log structure for the base $\sM_S^{min}$ the sub-log structure of $\sM_{S}$ generated by $\beta_{1}(e)$. Notice that a different choice of the chart $\beta_1$ will only alter the element $\beta_{1}(e)$ up to a unique invertible section of $\sM_{S}$. Thus, the sub-log structure $\sM_S^{min}$ is unique.  The obvious choice for our minimal log structure ${\sM_S'}^{min}$ is the direct sum $\sM_S^{min}\oplus_{\sO_{S}^{*}}\sN$. As a candidate for our minimal family of standard log points in $(X,\sM_X)$ we have the diagram
\begin{equation*}\label{logpt:local-min}
\xymatrix{
(S,{\sM'}^{min}) \ar[d]\ar[rr]^{\phi_{min}} && (X,\sM_X)\\
(S,\sM^{min}) & 
}
\end{equation*}
which comes with a natural log-map $f^{min}:(S,\sM_{S})\to (S,\sM^{min})$ simply by the construction of $\sM_S^{min}$. This family is certainly minimal at the point $s$, and is in fact a minimal family since minimality is an open condition and we need only further shrink $S$.  It is straightforward to check that the log map $f^{min}$ induces a morphism of families of standard log points in $(X,\sM_X)$ and is an isomorphism on the underlying schemes. Finally, uniqueness follows from the uniqueness of the sub-log structure $\sM^{min}$. 
\end{proof}

\subsection{The Stack $\wedge X$}\label{ss:wedgex}
We now construct the stack $\wedge X$ over the category $\mathfrak{Sch}$.  This stack parameterizes minimal families of standard log points in $(X,\sM_X)$.  We show that $\wedge X$ is isomorphic to an open substack of the fiber product $(\sA\times B\G_m)\times_{\mathcal{T}or_\C} \mathcal{T}or_{(X,\sM_X)}$, making $\wedge X$ an algebraic stack.

\begin{defn} Define a category $\wedge X$ fibered in groupoids over the category $\mathfrak{Sch}$ as follows:

\begin{enumerate}
\item[(a)] An object of the fiber $\wedge X(S)$ over the base scheme $S$ consists of a diagram
\[
\xymatrix{
(S,\sN_{S}') \ar[r]^\phi \ar[d] & (X,\sM_X)\\
(S,\sN_S) \ar[d] & \\
S\\}
\]
where $\phi$ gives a family of standard log points in $(X,\sM_{X})$ over $(S,\sN_{S})$ that is {\em minimal} in the sense of Definition \ref{defn:min}.

\item[(b)] An arrow sitting over a morphism of schemes $f:S\to T$ consists of a morphism of families of standard log points $F:(S,\sN_{S}') \to (T,\sN_{T}')$ over the strict morphism of log schemes $f:(S,\sN_S)\to (T,\sN_T)$ induced by $f$, forming a cartesian square and making the following diagram commutative:
\begin{equation*}\label{diag:log-pt-arrow}
\xymatrix{
&&(X,\sM_X)\\
(S,\sN_{S}') \ar[r]^F \ar[d] \ar@/^/[urr]^{\phi_{S}}& (T,\sN_{T}') \ar@/_/[ur]_{\phi_{T}}\ar[d]\\
(S,\sN_S) \ar[r]_f^{\text{strict}} & (T,\sN_T).\\}
\end{equation*}

\end{enumerate}

\end{defn}

We now give an alternate description of $\wedge X$ as an open substack of a fiber product of algebraic stacks.  Recall the definition in section \ref{ss:logstack} of the stack $\mathcal{T}or$ parameterizing fine, saturated log structures.  Furthermore, recall from the discussion in Example \ref{ss:families} that when taken with their trivial log structures, $\sA$ parameterizes families of DF(1) log structures and $B\G_m$ parameterizes families of standard log points.  Each comes with their respective universal log structure $\sM_\sA$ and $\sM_{B\G_m}$.  These log structures correspond to the inclusion $i:\sA\times B\G_m\to \mathcal{T}or_{\C}$.

Consider the fiber product $\sB:=(\sA\times B\G_m)\times_{\mathcal{T}or_\C} \mathcal{T}or_{(X,\sM_X)}$ given by the cartesian diagram
\[
\xymatrix{
\sB \ar[r] \ar[d] \ar[r] & \mathcal{T}or_{(X,\sM_X)} \ar[d]^{\text{Log}}\\
\sA\times B\G_m \ar[r]_i &\mathcal{T}or_{\C}. \\}
\]

The universal property of fiber product induces a morphism as follows:

\begin{defn}\label{defn:map-stack}
We define a morphism of fibered categories $\Phi:\wedge X \to \sB$.

Given an object in $\wedge X(S)$
 \[
\xymatrix{
(S,\sN_{S}') \ar[r]^{\phi_S} \ar[d] & (X,\sM_X)\\
(S,\sN_S) \ar[d] & \\
S\\}
\]
we obtain an object of $\sB(S)$ as follows: 
\begin{enumerate}
 \item the data of the log structure $\sN_S'\simeq \sN_S\oplus \sN$ is equivalent to a morphism $S\to \sA\times B\G_m$ since $\sN_S$ is a DF(1) log structure; 
 \item the arrow $\phi_{S}$ is equivalent to a morphism $S\to \mathcal{T} or_{(X,\sM_{X})}$.
\end{enumerate}
 Notice that the maps $S\to \mathcal{T} or_{\C}$ via $\sA\times B\G_m$ and $\mathcal{T} or_{(X,\sM_{X})}$ are identical, since they are given by $\sN_S\oplus \sN$. By the universal property of fiber products, this defines a morphism $S\to \sB$.  The morphism on arrows is defined similarly.  
\end{defn}

Denote by $\sM_{\sB}$ and $\sM_{\sB}'$ the log structures on $\sB$ pulled back from the canonical log structures on $\sA$ and $\sA\times B\G_{m}$ respectively. These log structures produce a family of standard log points $(\sB,\sM_{\sB}')\to (\sB,\sM_{\sB})$.  Furthermore, since the two compositions $\sB\to \sA\times B\G_{m}\to \mathcal{T} or_{\C}$ and $\sB\to \mathcal{T} or_{(X,\sM_{X})}\to \mathcal{T} or_{\C}$ coincide, there exists a map $\phi_\sB:(\sB,\sM_{\sB}')\to (X,\sM_{X})$. Thus $\sB$ naturally admits a family
\begin{equation*}\label{diag:univ-constr}
\xymatrix{
(\sB,\sM_{\sB}') \ar[r]^{\phi_\sB} \ar[d] & (X,\sM_{X}) \\
(\sB,\sM_{\sB})\ar[d]\\
\sB.&\\
}
\end{equation*}
of standard log points in $(X,\sM_{X})$.

Consider now any map $f:S\to \sB$ from a scheme $S$. The map $f$ corresponds to a unique family of standard log points in $(X,\sM_X)$ over $S$ given by the following pullback diagram:
\begin{equation*}\label{diag:univ-constr-1}
\xymatrix{
(S,\sM'_S)\ar[r] \ar[d]&(\sB,\sM_{\sB}') \ar[r] \ar[d] & (X,\sM_{X}) \\
(S,\sM_S)\ar[r]\ar[d]&(\sB,\sM_{\sB})\ar[d]\\
S\ar[r]&\sB.&\\
}
\end{equation*}
In this way, $\sB$ represents a fibered category over $\mathfrak{Sch}$ parameterizing families of standard log points in $(X,\sM_X)$ whose base log structure comes from a strict map to $(\sB,\sM_\sB)$.  Since the base log structure $\sM_S^{min}$ of a minimal family is constructed (in the proof of Proposition \ref{prop:min-cri}) as the sub-log structure generated by the single element $\beta_1(e)$, all minimal families are of this type.  It is not difficult to reformulate the definition of the map $\Phi$ using this dictionary for S-points of $\sB$.

Since minimality is an open condition, there must be an open substack $\sB'\subset\sB$ parametrizing minimal families of standard log points in $(X,\sM_{X})$ pulled back from $\sB$.  The morphism $\Phi$ factors through $\sB'$, providing a morphism $\Phi':\wedge X\to \sB'$. 

\begin{prop}
The functor $\Phi'$ is an equivalence. Hence, the stack $\wedge X$ is algebraic.
\end{prop}
\begin{proof}
That $\Phi'$ is both full and faithful follow from the above description of $\sB$ as a moduli stack of families of standard log points in $(X,\sM_X)$ pulled back from the family $(\sB,\sM_\sB')\to(\sB,\sM_\sB)$ since the morphisms are the same in both categories.  Essential surjectivity is obvious from the description of the map.  Consider a minimal family of standard log points in $(X,\sM_X)$ over $S$:
\[
\xymatrix{
(S,\sM_{S}') \ar[r]^{\phi_S} \ar[d] & (X,\sM_{X}) \\
(S,\sM_{S})\ar[d]&\\
S.\\
}
\]
The log structure $\sM_S'\cong\sM_{S}\oplus\sN$ is equivalent to a map $S\to \sA\times B\G_{m}$, and the log map $\phi_S$ is equivalent to a map $S\to \mathcal{T}or_{(X,\sM_{X})}$. The two compositions $S\to \sA\times B\G_{m}\to \mathcal{T} or_{\C}$ and $S\to \mathcal{T} or_{(X,\sM_{X})}\to \mathcal{T}or_{\C}$ coincide. Since the family is minimal, this is equivalent to a map $S\to \sB$ factoring through $\sB'$.
\end{proof}

\begin{rem}
The stack $\wedge X$ comes equipped with two canonical log structures $\sM_{\wedge X}$  and $\sN_{\wedge X}$ coming from the log structures $\sM_\sA$ and $\sM_{B\G_m}$ respectively.
\end{rem}

\subsection{The Category $(\wedge X,\sM_{\wedge X})$ fibered over $\mathfrak{LogSch}^\text{fs}$}

We will now make explicit the connection between the stack $\wedge X$ over $\mathfrak{Sch}$ and the groupoid fibration $\wedge' X$ over $\mathfrak{LogSch}^\text{fs}$ from Definition \ref{def:wedgeprime}.  We begin by discussing the universal structures on $\wedge X$.

The construction of the evaluation stack $\wedge X$ gives two natural log structures.  The first, $\sM_{\wedge X}$, is induced by restriction from the log structure $\sM_{\sA}$ on $\sA$.  The second, $\sN_{\wedge X}$, is induced by the log structure on $B\G_m$.  The pairs $(\wedge X,\sM_{\wedge X}\oplus\sN_{\wedge X})$ and $(\wedge X,\sM_{\wedge X})$ fit into a family of standard log points:
 \[
\xymatrix{
(\wedge X,\sM_{\wedge X}\oplus\sN_{\wedge X}) \ar[r] \ar[d] & (X,\sM_X)\\
(\wedge X,\sM_{\wedge X}). & \\}
\]

This family is a universal object for the evaluation stack, the map to $\wedge X$ given simply by forgetting the log structures and the map to $(X,\sM_X)$.  A morphism $f:S\to\wedge X$ is equivalent to a pull-back diagram
 \[
 \tag{$\ast$}
\xymatrix{
(S,\sM_S\oplus \sN) \ar[r] \ar@/^1pc/[rr]^{\phi_S} \ar[d] & (\wedge X,\sM_{\wedge X}\oplus\sN_{\wedge X}) \ar[r] \ar[d] & (X,\sM_X)\\
(S,\sM_S) \ar[d] \ar[r]^{f'}& (\wedge X,\sM_{\wedge X}) \ar[d]& \\
S\ar[r]^f & \wedge X. &\\}
\]

The map $f'$ above is strict, and the family of standard log points
 \[
\xymatrix{
(S,\sM_S\oplus\sN) \ar[r]^{\phi_S} \ar[d] & (X,\sM_X)\\
(S,\sM_S) & \\}
\]
corresponding to $f$ is a minimal family in the sense of definition \ref{defn:min} precicely because of the strictness of $f'$.

\begin{prop}\label{lem:family-wedgex}
The data of a morphism $(S,\sM_S)\to (\wedge X,\sM_{\wedge X})$ is equivalent to giving a family of standard log points over $(S,\sM_S)$ in $(X,\sM_X)$.  Thus, the log stack $(\wedge X,\sM_{\wedge X})$ represents the fibered category over $\mathfrak{LogSch}^\text{fs}$ parametrizing families of standard log points in $(X, \sM_X)$.  In other words, the categories $(\wedge X, \sM_{\wedge X})$ and $\wedge' X$ are equivalent as groupoid fibrations over $\mathfrak{LogSch}^\text{fs}$.
\end{prop}
\begin{proof}
A morphism $h': (S,\sM_S)\to(\wedge X,\sM_{\wedge X})$ detrmines a commutative diagram
\[
\xymatrix{
(S,\sM_S\oplus \sN) \ar[r] \ar@/^1pc/[rr]^{\phi_S} \ar[d]_g & (\wedge X,\sM_{\wedge X}\oplus\sN_{\wedge X}) \ar[r] \ar[d] & (X,\sM_X)\\
(S,\sM_S) \ar[r]^{h'}& (\wedge X,\sM_{\wedge X})&}
\]
where the bottom left square is cartesian.  Unlike the top half of the diagram $(\ast)$, the map $h'$ is not necessarily strict. The pair of maps $g$ and $\bar h$ determine a family of standard log points in $(X,\sM_X)$. 

Inversely, let $\phi_S:(S,\sM_S') \to (X,\sM_X)$ be a family of standard log points in $(X,\sM_X)$ over $(S,\sM_S)$.  There exists a unique minimal family $\phi^{min}:(S,{\sM_S'}^{min})\to(X,\sM_X)$ sitting over $(S,\sM_S^{min})$ and a map $f^{min}:(S,\sM_{S})\to (S,\sM_S^{min})$ fitting into a morphism of families, as stated in Proposition \ref{prop:min-cri}.  Since $\phi^{min}$ in turn sits over the scheme $S$, this produces a map $f:S\to\wedge X$ and an extended diagram :
 \[
\xymatrix{
(S,\sM_S')\ar[d]\ar[r]&(S,{\sM_S'}^{min}) \ar[r]  \ar[d] & (\wedge X,\sM_{\wedge X}\oplus\sN_(S,\sM_S){\wedge X}) \ar[r] \ar[d] & (X,\sM_X)\\
(S,\sM_S)\ar[r]_{f^{min}}&(S,\sM_S^{min}) \ar[d] \ar[r]_{f'}& (\wedge X,\sM_{\wedge X}) \ar[d]& \\
&S\ar[r]_f & \wedge X. &\\}
\]
The desired map $h':(S,\sM_S)\to(\wedge X,\sM_{\wedge X})$ is determined by the composition of $f^{min}$ and $f'$, and the uniqueness of $f^{min}$. 
\end{proof}


\subsection{Contact order decomposition of $\wedge X$}  In this section we give a stratification of $\wedge X$ indexed by $\N$.  The components of the stratification correspond to possible contact orders of marked points.  Consider the family $\phi_S:(S,\sM_{S}\oplus\sN_{S}) \to (S,\sM_{S})$ of standard log points in $(X,\sM_{X})$ over $(S,\sM_S)$, and let s be a geometric point of $S$.  We have a map on the level of characteristics
\begin{equation*}\label{diag:contact-order}
\phi_S^{*}\overline{\sM}_{X}\stackrel{\phi_S^\flat}{\rightarrow} \overline{\sM}_{S}\oplus\overline{\sN}_{S} \rightarrow \overline{\sN}_{S}
\end{equation*}
where the second arrow is given by the natural projection. Assume that the image of $s$ in $X$ lies in the locus where $\sM_{X}$ non-trivial (i.e. is mapped to the relative divisor). Denote by $\delta$ and $\sigma$ the generators of $\overline{\phi_S^\ast(\sM_{X,s})}$ and $\overline{\sN}_{S,s}$ respectively. Then the above composition restricts on the stalks at $s$ to $\delta\mapsto c\cdot \sigma$ for some integer $c\in\N$. 

\begin{defn}
The integer $c$ is called the {\em contact order} of the standard log point $\phi$ over the geometric point $s$. When the image of $s$ in $X$ lies in the locus with $\sM_{X}$ trivial, we define the contact order $c=0$. 
\end{defn}

This definition corresponds exactly to the contact order of the marked points of a minimal log stable map.  The following lemma shows that the contact order remains constant along a family.  This fact provides our stratification.

\begin{lem}\label{lem:contact}
In a minimal family of standard log points in $(X,\sM_{X})$ over a scheme $S$, the points whose fibers have fixed contact order $c$ form an open subscheme of $S$.
\end{lem}
\begin{proof}
Again, denote our family by $\phi_S:(S,\sM_S')\to (X,\sM_X)$ over $(S,\sM_S)$.  Assume that the fiber over the point $s\in S$ has contact order $c$, and $\sM_{S}'\cong \sM_{S}\oplus\sN_{S}$.  As above, we have the composition $\phi^{*}(\overline\sM_{X})_{s}\to \sM_{S,s}'/\sM_{S,s}\cong \overline{\sN}_{S}$ is given by $\delta\mapsto c\cdot \sigma$.  By \cite{O03} Lemma 3.5, this generalizes to nearby points of $s$. Thus fibers having contact order $c$ is an open condition on the base. 
\end{proof}

\begin{prop}
We have the following disjoint union of stacks
\[\wedge X = \coprod_{c\in \N}\wedge_{c} X\]
where $\wedge_{c} X$ is the stack parameterizing minimal families of standard log points in $(X,\sM_{X})$ with contact order $c$.
\end{prop}
\begin{proof}
By Lemma \ref{lem:contact}, the value of the contact order along the fiber of a family defines a continuous map $\wedge X\to \N$.  Since $\N$ is discrete, this map prescribes the stratification $\wedge X = \coprod_{c\in \N}\wedge_{c} X$.  The stack $\wedge_cX$ is exactly the pre-image of $c$.
\end{proof}

In order to make use of $\wedge X$ as an evaluation space for logarithmic stable maps, we need to understand the basic structure of the components $\wedge_cX$.  We end this section with such an analysis, beginning with the following easy proposition.

\begin{prop}
$\wedge_{0}X = X\times B\G_{m}$
\end{prop}
\begin{proof}
In a family of standard log points in $(X,\sM_{X})$ over a scheme $S$ with contact order $0$, we have $\sM_{S}\cong \phi_S^{*}\sM_{X}$. The log structure $\sN_{S}$ is given by a map $S\to B\G_{m}$, thus such families are equivalent to a map $S\to X\times B\G_{m}$. 
\end{proof}

\begin{cor}\label{cor:triv}
If the target $(X,\sM_X)$ comes equipped with the trivial log strucure $\sM_X\cong \sO_X^\ast$, then $\wedge X = X\times B\G_{m}$.
\end{cor}
\begin{proof}
In this case $\wedge X =\wedge_{0} X$.
\end{proof}

Recall from section \ref{ss:families} the interpretation of $\sA$ as the stack associating to a scheme $S$ the groupoid of pairs $(L,s)$ where $L$ is a line bundle and $s\in H^{0}(L^\vee)$. The substack $B\G_m\subset\sA$  associates to $S$ the groupoid of pairs $(L,0)$, where $L$ is a line bundle, and $0$ is the zero section. 

\begin{defn}\label{defn:tangency-map}
Define a map
\begin{equation}
\nu_{c} : \sA\times B\G_{m} \to B\G_{m} \ \ \ \big((L,s),(L',0)\big)\mapsto (L\otimes L'^{\otimes c}, 0).
\end{equation}
This map is defined fiber-wise, where $L$ and $L'$ are line bundles over the base $S$.
\end{defn}

\begin{rem}
We should, in fact, view the map $\nu_{c}$ as sending sections $s$ and $0$ to $s\cdot 0^{\otimes c}$.
\end{rem}

\begin{lem}\label{lem:evaluation-map}
The map $\nu_{c}$ induces a morphism of log stacks with their natural log structures.
\end{lem}
\begin{proof}
By the discussion in Example \ref{ss:families}, the map on the level of line bundles with sections induces a map of corresponding $\G_{m}$-torsors, hence a map of sheaves of monoids $\nu_{c}^{\flat}:\nu_{c}^{*}\sM_{B\G_{m}}\to \sM_{\sA}\oplus\sM_{B\G_{m}}$ given by $\delta\mapsto e + c\cdot \sigma$. Here $\delta, e$, and $\sigma$ are the generators of $\nu_{c}^{*}\sM_{B\G_{m}}, \sM_{\sA}$, and $\sM_{B\G_{m}}$ respectively. One can check that $\nu_{c}^{\flat}$ gives a map of the corresponding log structures.
\end{proof}

Let $D\subset X$ be the relative divisor on $X$ corresponding to $\sM_X$, i.e. the locus in $X$ where $\sM_X$ is non-trivial.  The locus $D$ has a natural closed scheme structure given locally by the generator of $\sM_{X}$. Let $\sM_{D}=\sM_{X}|_{D}$ denote the log structure on $D$.  This log structure is naturally isomorphic to the log structure induced by $\sO(-D)|_{D}$.  
Since $\sM_D$ is itself also DF(1) (in fact, it is generically rank 1), this induces a natural map $f:(D,\sM_D) \to (B\G_{m},\sM_{B\G_m})$. The cartesian diagram of log stacks
\begin{equation}\label{diag:c-log-evaluation}
\xymatrix{
(I_{c},\sM_{I_{c}}') \ar[r] \ar[d] & (\sA\times B\G_{m},\sM_{\sA\times B\G_{m}}) \ar[d]^{\nu_{c}} \\
(D,\sM_{D}) \ar[r]_f&  (B\G_{m},\sM_{B\G_{m}}).
}
\end{equation}
gives a fiber product description of the component $\wedge_c X$.

\begin{prop}
$\wedge_{c} X \cong I_{c}$
\end{prop}
\begin{proof}
We use the universal property of the above cartesian diagram. In fact, $\sM_{\wedge X}$ is given by the composition $I_{c}\to \sA\times B\G_{m} \to \sA$ and $\sN_{\wedge X}$ is given by the composition $I_{c}\to \sA\times B\G_{m} \to B\G_{m}$ (the second arrow is the projection, not $\nu_{c}$).  Denote the log structures on $I_c$ corresponding to these compositions by $\sM_{I_{c}}$ and $\sN_{I_{c}}$ respectively. Then we have $\sM_{I_{c}}'=\sM_{I_{c}}\oplus_{\sO_{I_{c}}} \sN_{I_{c}}$. Consider the diagram:
\begin{equation}\label{diag:univ-min-pt}
\xymatrix{
(I_{c},\sM_{I_{c}}') \ar[rr] \ar[d] && (D,\sM_{D}) \ar@{^{(}->}[rr]^{strict} && (X,\sM_{X})\\
(I_{c},\sM_{I_{c}}).
}
\end{equation}
This induces a family of standard log points in $(X,\sM_{X})$. By the strictness of $f'$ in (\ref{diag:c-log-evaluation}) and the description of $\nu_{c}^{\flat}$ in the proof of Lemma \ref{lem:evaluation-map}, the above diagram in fact gives a minimal family of standard log points in $(X,\sM_{X})$ with contact order $c$.

Consider any family of minimal log points in $(X,\sM_{X})$ with contact order $c$:
\begin{equation}\label{diag:test-log-pt}
\xymatrix{
(S,\sM_{S}') \ar[rr]^{f_S} \ar[d] && (X,\sM_{X}) \\
(S,\sM_{S}).
}
\end{equation}
The log structure $\sM_{S}'$ induces a strict log map $(S,\sM_{S}')\to (\sA\times B\G_{m},\sM_{\sA\times B\G_{m}})$. Composing with $\nu_{c}$ determines a map $\phi:(S,\sM_{S}')\to (B\G_{m},\sM_{B\G_{m}})$. On the other hand, since $c>0$, the map $f_{S}$ factor through $(D,\sM_{D})$. This induces another map $\phi':(S,\sM_{S}')\to (B\G_{m},\sM_{B\G_{m}})$ which coincides with $\phi$ by the description of $\nu_{c}$. Thus, we obtain a unique map $g:S\to I_{c}$, such that (\ref{diag:test-log-pt}) is the pull-back of (\ref{diag:univ-min-pt}). The statement follows from the moduli interpretation of $\wedge_{c}X$.
\end{proof}

\subsection{Cohomology of $\wedge X$}  Recall from \cite{E10} that the cohomology ring of a quotient stack is given by the equivariant cohomology of the corresponding group action.  The cohomology of the components $\wedge_c X$ can be described combinatorially in this way.

Let $N_{D/X}$ be the normal bundle of $D$ in $X$, and let $\sN_{D/X}$ be the DF($1$) log structure over $D$ induced by $N_{D/X}$ with the zero section.  The diagram of log stacks above induces the following cartesian diagram of stacks:
\[
\xymatrix{
\wedge_{c} X \ar[r] \ar[d] & \sA\times B\G_{m} \ar[d]^{\nu_{c}} \\
D\ar[r]_f&  B\G_{m}.
}
\]
The map $f$, as above, is induced by $\sN_{D/X}$.  The stack $\sA\times B\G_m$ can be written as the quotient $[\A^1/\G_m^2]$ where the action of the first component is given by multiplication and is trivial for the second component.  The corresponding weights of these actions give respective equivariant parameters $s$ and $t$.  Let $\gamma$ denote the first chern class $c_1(N_{D/X}) \in H^\ast(D)_\Q$.  The cohomology of $\wedge_c X$ is described in the following proposition.

\begin{prop} $H^\ast(\wedge_c X)_\Q = H^\ast(D)_\Q[s,t]/(s+ct-\gamma)$
\end{prop}
\begin{proof}
The torus equivariant cohomology of $\A^1$ given by the above action is simply $H^\ast(B\G_m^2)$.  This is the polynomial ring in the variables $s$ and $t$ representing the weights of the action.  Over a base $S$, let $\sM_S\oplus\sM_S'$ and $\sN_S$ denote the log structures corresponding to $S\to\sA\times B\G_{m}$ and $S\to B\G_m$ respectively.  Set local generators $e\in\overline{\sM}_S$, $\sigma\in\overline{\sM}_S'$ and $\delta\in\overline{\sN}_S$. The map $\nu_c$ corresponds to sending $\delta$ to $e+c\sigma$.  
The result follows.
\end{proof}

\section{Generalization to fs log schemes} \label{ss:dfncase}  

\subsection{The DF(n) case}
We now consider the more general case of a log-smooth target scheme $X$ endowed with a DF log structure $\sM_X$ of arbitrary rank. Extending the construction of $\wedge X$ to this case provides an evaluation space for minimal logarithmic stable maps as constructed in \cite{AC10} for DF($n$) log schemes.  Following this, we will extend even further in section \ref{ss:fslog} to any fine, saturated target log scheme. 


A DF($n$) log structure $\sM_X$ is globally presented by a morphism $\N^r\to\overline{\sM}_X$ that lifts locally to a chart. Let $i=1,\ldots, n$ index the $r$ copies of $\N$ and $j_i:\N\hookrightarrow\N^r$ the incusion of the $i$-th component. The map 
\[
\N\overset{j_{i}}{\hookrightarrow} \N^{r} \to \overline{\sM}_{X}
\]
induces a DF(1) sub-log structure $\sM_{i}\subset \sM_{X}$, defining DF(1) log schemes $(X_i,\sM_{i})$ for each $i$ (with $X_i=X$). $(X,\sM_X)$ can be written as a fibered product
\[
(X,\sM_{X})= (X_1,\sM_{X_1})\times_{X}\cdots\times_{X}(X_r,\sM_{X_r}), 
\]
which simply says 
\[
(X,\sM_X)=\mathop{\varprojlim}\limits_i (X_i,\sM_{X_i}).
\]
This fact allows us to take advantage of our construction for the $\wedge X_{i}$.

Recall from Corollary \ref{cor:triv} that when $X$ carries the trivial log structure $\sO_X^\ast$, the evaluation stack is isomorphic to $X\times B\G_m$.  The canonical map $(X_i,\sM_{X_i})\to X$ induces a map 
\[(\wedge X_i,\sM_{\wedge X_{i}}) \to (X\times B\G_{m}, \sM_{B\G_{m}}).\]
Consider the fiber product
\[
(\wedge X,\sM_{\wedge X}) := (\wedge X_1,\sM_{\wedge X_1})\times_{(X\times B\G_m,\sM_{B\G_m})}\cdots\times_{(X\times B\G_m,\sM_{B\G_m})}(\wedge X_r,\sM_{\wedge X_r}).
\]
This is equivalent to 
\[(\wedge X,\sM_{\wedge X}) = \mathop{\varprojlim}\limits_i(\wedge X_i,\sM_{\wedge X_i}).\]

\begin{prop}\label{prop:DFn-rep}
The log algebraic stack $(\wedge X,\sM_{\wedge X})$ represents the fibered category $\wedge'X$.
\end{prop}
\begin{proof}
Theorem 2.6 in \cite{AC10} is the main argument used to extend the theory of minimal log stable maps to the DF($n$) case.  Our situation is a special case of this:  for the category $\text{Sec}_{C/B}(Z/C)$ fibered over $\mathfrak{LogSch}^{\text{fs}}_{(B,\sM_B)}$ to which the theorem applies, simply take $B$ to be $X\times B\G_m$, $C$ to be the universal family $(X\times B\G_m, \sM_{B\G_m})$, $W:=(X,\sM_X)\times B$, and $Z$ to be the fiber product of $C$ with $W$ over $(X\times B)\times_{B}B$.
\end{proof}

Using this fibered product construction, we obtain a universal diagram
\begin{equation}\label{diag:univ-DFn}
\xymatrix{
(\wedge X,\sM_{\wedge X}\oplus \sN_{\wedge X}) \ar[d] \ar[rr]^{f} && (X,\sM_X) \\
(\wedge X,\sM_{\wedge X}).
}
\end{equation}

\begin{defn}\label{defn:DFn-minimal}
A family of standard log points in $(X,\sM_X)$ over $(S,\sM_{S})$ is called \emph{minimal} if it is the pull-back of (\ref{diag:univ-DFn}) along a strict map $(S,\sM_{S})\to (\wedge X,\sM_{\wedge X})$.
\end{defn}

\begin{rem}
It is enough to check strictness at geometric points. Hence a family of standard log points in $(X,\sM_X)$ is minimal if and only if each geometric fiber is minimal.
\end{rem}

As in the DF(1) case, we can give a combinatorial description of minimality. 

\begin{prop}
Consider a family
\[
\xymatrix{
(S,\sM_{S}\oplus\sN_{S}) \ar[rr]^{f} \ar[d] && (X,\sM_X)\\
(S,\sM_{S})
}
\]
where $S$ is a geometric point. It is minimal in the sense of Definition \ref{defn:DFn-minimal} if and only if 
\begin{enumerate}
 \item $\overline{\sM_{S}}\cong \N^{m}$ for some $m\leq r$;
 \item for each irreducible element $e\in \overline{\sM_{S}}$, there exists a unique irreducible element $\delta\in \overline{f^{*}\sM_{X}}$ whose image under the composition $\overline{f^{*}\sM_{X}} \to \overline{\sM_{S}}\oplus\overline{\sN_{S}}\to \overline{\sM_{S}}$ is $e$.
\end{enumerate}
\end{prop}
\begin{proof}
Denote by $\phi_{i}: (X,\sM_X) \to (X_{i},\sM_{X_i})$ the $i$-th projection. The composition $\phi_{i}\circ f$ produces a family of standard log point in $(X_{i},\sM_{X_i})$. Denote by $f_i$ the minimal family associated to $\phi_{i}\circ f$ provided by Proposition \ref{prop:min-cri}, 
\[
\xymatrix{
(S,\sM^{min}_i\oplus\sN_{S}) \ar[rr]^{f_{i}} \ar[d] && (X_i,\sM_{X_i})\\
(S,\sM^{min}_i),
}
\]
with the natural map $(S,\sM_{S})\to (S,\sM^{min}_{i})$. Since $(\wedge X,\sM_{\wedge X})$ is constructed as a fibered product, by Proposition \ref{prop:DFn-rep} we have a canonical map
\[\psi:(S,\sM_{S}) \to (S,\sM_{1})\times \cdots \times (S,\sM_{r}).\]
Since $S$ is a geometric point, $\psi$ is an isomorphism of the underlying schemes. 

To complete the proof, assume first that the family in the statement of the proposition is minimal. Then $\psi$ is a isomorphism. Since $\sM^{min}_{i}$ is DF(1) log structure for all $i$, condition (1) holds.  Consider the composition $\overline{f^{*}\sM_{X}} \to \overline{\sM_{S}}\oplus\overline{\sN_{S}}\to \overline{\sM_{S}}$. Note that it can be constructed as the product of $\overline{f_{i}^{*}\sM_{X}} \to \overline{\sM^{min}_{i}}\oplus\overline{\sN_{S}}\to \overline{\sM^{min}_{i}}$ for $i=1,2,\cdots r$. Thus condition (2) follows.

For the other direction, assume $\overline{\sM_{S}}$ satisfies the two conditions in the statement. Then one can easily check that $\psi$ induces an isomorphism of characteristics, hence is an isomorphism.
\end{proof}

\subsection{The case of fs log schemes} \label{ss:fslog}
Finally, we consider the case of an arbitrary fine, saturated log
scheme $(X,\sM_{X})$, as in Theorem \ref{thm:main}.

\begin{proof}[Proof of Theorem \ref{thm:main}]
The statement is local on $X$, thus we can shrink $X$ and assume that
there is a chart $P\to \sM_{X}$ with $P$ fine and saturated. The
following lemma finishes the proof.  
\end{proof}

\begin{lem}
Assume there is a map $P\to \overline{\sM}_{X}$ from a fine, saturated monoid $P$ that locally lifts to a chart. Then the fibered category $\wedge'X$ is represented by a log stack $(\wedge X,\sM_{\wedge X})$.
\end{lem}
\begin{proof}
A log structure with this property, called a \emph{generalized DF log
  structure}, is shown in \cite{AC10} to have a useful structure which
we describe briefly here. By 2.1.9(7) in Chapter 1 of \cite{OgusText},
we can write $P=
\mathop{\varinjlim}\limits(\N^{a}\rightrightarrows\N^{b})$ for
non-negative integers $a,b$.  Since $P\to\overline{\sM}_X$ locally
lifts to a chart, the compositions $\N^{a}\to P\to\sM_X$ and
$\N^{b}\to P\to\sM_X$ induce respective DF($a$) and DF($b$) log
structures $\sM_{a}$ and $\sM_{b}$ on $X$ fitting into a limit diagram
of log schemes
$$(X,\sM_X)= \mathop{\varprojlim}\bigg((X_{b},\sM_{b})\rightrightarrows
(X_{a},\sM_{a})\bigg),$$     
where $X_a=X_b=X$. Finally, a similar application of Theorem 2.6 from
\cite{AC10} as used in our proof of Proposition \ref{prop:DFn-rep}
shows that $\wedge'X$ is represented by  
\[\mathop{\varprojlim}\bigg(\wedge(X_{b},\sM_{b})\rightrightarrows   
\wedge(X_{a},\sM_{a})\bigg).
\]
\end{proof}

\section{Logarithmic Gromov-Witten theory}\label{s:loggw}

\subsection{}Our construction of $\wedge X$ with its universal diagram
\[
\xymatrix{
(X,\sM_{\wedge X}\oplus\sN_{\wedge X}) \ar[r] \ar[d] & (X,\sM_X)\\
(\wedge X,\sM_{\wedge X}) \ar[d]& \\
\wedge X. &\\}
\]
mimics the situation for $\sK_\Gamma(X,\sM_X)$ described in the introduction.  The stack $\wedge X$ provides the evaluation space for $\sK_\Gamma(X,\sM_X)$, allowing us to define evaluation maps and the GW invariants they produce in this setting. Restricting the universal log stable map $(\mathfrak{C},\sM_{\mathfrak{C}})\to(X,\sM_X)$ over $(\sK_\Gamma(X,\sM_X),\sM_{\sK_\Gamma(X,\sM_X)})$ to the $i$-th marked point gives a morphism 
\[
\text{ev}_i:(\sK_\Gamma(X,\sM_X),\sM_{\sK_\Gamma(X,\sM_X)})\to(\wedge X,\sM_{\wedge_X}).
\]
of fibered categories.  We define these to be our evaluation morphisms. 

\begin{defn}\label{def:eval}
For $i=1,\ldots,n$, define a morphism 
\[
\text{ev}_i:(\sK_\Gamma(X,\sM_X),\sM_{\sK_\Gamma(X,\sM_X)})\to(\wedge X,\sM_{\wedge_X})
\] 
of categories fibered over $\mathfrak{LogSch}^\text{fs}$ as follows:

On the level of objects, a morphism $(S,\sM_S)\to (\sK_\Gamma(X,\sM_X),\sM_{\sK_\Gamma(X,\sM_X)})$ corresponds to a family of log stable maps 
 \[
\xymatrix{
(C,\sM_C) \ar[r] \ar[r]^{f} \ar[d] & (X,\sM_X)\\
(S,\sM_S)& }
\]
over a log scheme $(S,\sM_S)$.  Let $\Sigma_i\subset C$ be the image of the section $\sigma_i:S\to C$ corresponding to the $i$-th marked point.  Restricting to $\Sigma_i$ provides a locus in $C$ isomorphic to $S$ where $\sM_C$ has the structure of a standard log point.  This gives a family of standard log points in $(X,\sM_X)$
\[
\xymatrix{
(\Sigma_i,\sM_{C}|_{\Sigma_i}) \ar[r]^{f'|_{\Sigma_i}} \ar[d] & (X,\sM_{X}) \\
(S,\sM_{S})&
}
\]
which in turn corresponds to a morphism $(S,\sM_S)\to(\wedge X,\sM_{\wedge X})$.

The morphism on the level of arrows is defined similarly.
\end{defn}

\begin{prop}\label{prop:vfc}
Assume that the generalized DF pair $(X,\sM_{X})$ is log smooth. Then the stack of minimal log stable maps admits a virtual fundamental class $\virclass{\sK_\Gamma(X,\sM_X)}$.
\end{prop}
\begin{proof}
Denote by $\sT or_{\C}$ the stack parameterizing fs log structures over $\C$. Note that we can build a commutative log diagram as follows:
\begin{equation}\label{diag:total-diag}
\xymatrix{
(\mathfrak{C},\sM_\mathfrak{C}) \ar@/^/[rrrd] \ar[rd] \ar@/_/[rdd]_{} &&&\\
&(X_\sK,\sM_{X_\sK}) \ar[r] \ar[d] & (X_\sT,\sM_{X_\sT}) \ar[d] \ar[r] & (X,\sM_X) \ar[d]\\
&\sK_\Gamma(X,\sM_X) \ar[r]^{\ \ \ \alpha} & \sT or_{\C} \ar[r] & pt,
}
\end{equation}
where $(X_\sK,\sM_{X_\sK})$ and $(X_\sT,\sM_{X_\sT})$ are the respective fiber products log stacks, and the arrow $\alpha$ is induced by the canonical log structure on $\sK_\Gamma(X,\sM_X)$. An identical argument to Section 4.1 of \cite{LogDeg} implies that there exists a perfect obstruction theory $E.\to \mathbb{L}_{\sK_\Gamma(X,\sM_X)/\sT or_{\C}}$ for the map $\alpha:\sK_\Gamma(X,\sM_X)\to \sT or_{\C}$. The complex $E.$ determines a vector bundle stack $\mathbb{E}$ over $\sK_\Gamma(X,\sM_X)$, and the pair $(\alpha,\mathbb{E})$ satisfy Condition 2.8 (i.e., condition $\star$) from \cite{M08}.  Thus we may use Manolache's refined pullback $\alpha^!_{\mathbb{E}}:A_\ast(\sT or_{\C})\to A_\ast(\sK_\Gamma(X,\sM_X))$.  

Note that $\sT or_{\C}$ is of pure dimension $0$ and stratified by global quotients (as in Definition 4.5.3 of \cite{K99}).  Thus there is a fundamental class $[\sT or_{\C}]$, and we have $\virclass{\sK_\Gamma(X,\sM_X)}= \alpha^!_\mathbb{E}[\sT or_\C]$.


\end{proof}

 This virtual class and the evaluation maps $\text{ev}_i$ provide the necessary ingredients for logarithmic GW invariants, which we now define.

\begin{defn} \label{def:gw-inv} For
  each $i=1,\ldots,n$, fix cohomology classes $\gamma_i\in
  H^\ast(\wedge X)_\Q$.  Pulling the $\gamma_i$ back along the
  evaluation maps $\text{ev}_i$, define the logarithmic GW invariant
  $\left<\gamma_1,\ldots,\gamma_n\right>^{(X,\sM_X)}_{\Gamma}$ by the
  product 
    \[
    \left<\gamma_1,\ldots,\gamma_n\right>^{(X,\sM_X)}_{\Gamma} := \left(\prod_{i=1}^n \text{ev}_i^\ast\gamma_i \right)\cap \left[\sK_\gamma(X,\sM_X)\right]^{\text{vir}}.
    \]
\end{defn}    

%

\bibliographystyle{amsalpha}             
\bibliography{logpoints}       
\end{document}